\documentclass{amsart}
\usepackage{amssymb}
\newtheorem{theorem}{Theorem}[section]
\newtheorem{lemma}[theorem]{Lemma}

\theoremstyle{definition}
\newtheorem*{definition}{Definition}

\newtheorem{remark}[theorem]{Remark}
\DeclareMathOperator{\rank}{rank}

\numberwithin{equation}{section}

\begin{document}

\def\bR{{\mathbf{R}}}
\def\bC{{\mathbf{C}}}
\def\bZ{{\mathbf{Z}}}
\def\rp2{{\mathbf{R}\mathrm{P}^2}}
\def\cp2{{\mathbf{C}\mathrm{P}^2}}
\def\ctm{{\mathbf{C}\otimes TM}}
\def\tri{totally real immersion~}
\def\bRtwoN{{\mathbf{R}^{2N}}}
\def\onejet{J^1(M,\mathbf{R}^{2N})}

\title[Optimality for totally real immersions and independent mappings]
{Optimality for totally real immersions and independent mappings of manifolds into $\bC^N$}

\author{Pak Tung Ho}
\address{Department of Mathematics, Sogang University, Seoul 121-742, Korea}
\email{ptho@sogang.ac.kr}

\author{Howard Jacobowitz}
\address{Department of Mathematical Sciences, Rutgers University, Camden, New Jersey, USA}
\email{jacobowi@camden.rutgers.edu}

\author {Peter Landweber}
\address{Department of Mathematics, Rutgers University, Piscataway, NJ 08854, USA}
\email{landwebe@math.rutgers.edu}


\date{March 13, 2012}


\begin{abstract} The optimal target dimensions are determined for 
totally real immersions and for
independent mappings into complex affine spaces. Our arguments are similar to those given by Forster, but we use orientable manifolds as far as possible and so are able to obtain improved results for orientable manifolds of even dimension.  This leads to new examples showing that the known  immersion and submersion dimensions for holomorphic mappings from Stein manifolds to affine spaces are best possible. 
\end{abstract}

\maketitle

\section{Introduction}

The following two theorems are easily proved by counting dimensions and applying the Thom Transversality Theorem.  See \cite{lecturenotes} for a full discussion, and the appendix for a brief account. 

\begin{definition} 
A smooth immersion $f:M \to\bC^N$
is called \textit{totally real} if $f_*(TM)$ does not contain any
complex line. (All manifolds considered are assumed to be smooth and second countable.)  This is equivalent to requiring that
$$f_*(TM)\cap i(f_*(TM))=\{0\}.$$
\end{definition}

\begin{theorem}\label{tri}
There exists a totally real immersion of each $n$-dimensional manifold $M$ into $\bC^N$,
provided $N \geq [\frac{3n}{2}]$.
\end{theorem}

\smallskip
\begin{definition}
A smooth mapping 
$F: M \to \bC^N$ is called \emph{independent} if its coordinate functions $F_1, \ldots, F_N$ satisfy 
$dF_1\wedge\cdots\wedge dF_N \neq 0$ at all points of $M,$ in which case the coordinate functions are also called \emph{independent.}
\end{definition}

\begin{theorem}\label{functions}
For every manifold $M$ of dimension $n$ there is a smooth
map $F:M\to \bC^N$ whose coordinate functions are independent, provided $N \leq [{\frac{n+1}{2}}]$.  
\end{theorem}

Theorems obtained using transversality, such as these, are often, but not always, optimal in the sense that the target dimension cannot be decreased (in a case such as Theorem \ref{tri}) or increased (in a case such as Theorem \ref{functions}).  Recall that a transversality argument implies that every $n$-dimensional manifold $M$ has an immersion into $\bR^{2n}$, but a more delicate argument due to Whitney decreases $2n$ to $2n-1$ for $n>1.$

\medskip
The aim of this paper is to prove the optimality of the theorems stated above, by constructing and examining suitable simple examples (of closed manifolds) in all positive dimensions.  The arguments, which are very similar to those due to Forster \cite{Fo2}, are presented in the next two sections.  In addition, in \S4 we prove slightly stronger (optimal) results for orientable manifolds having dimension of the form $4k+2,$ and also for orientable manifolds of dimension $4k$ under the assumption that the top Pontryagin class (or top dual Pontryagin class) vanishes.  

In the final section we compare our results to those for holomorphic immersions and submersions of Stein manifolds proved by Forster \cite{Fo2} and Forstneri\v{c} \cite{For1}.  See Chapter 8 of the recent book by Franc Forstneri\v{c} \cite{For2} for a full account of these results.  

The appendix outlines how to prove Theorems \ref {tri} and \ref {functions} using simple transversality arguments.


\section{Optimality for totally real immersions}
We shall show that the target dimensions in Theorem \ref{tri} cannot be decreased. 
This will be accomplished by finding manifolds $M^{2n}$ and $M^{2n+1}$ in all positive dimensions so that

\begin{itemize}

	\item 
	$M^{2n}$ does not admit a  \tri  into $\bC ^N$ for $N=3n-1$, and

	\item 
	$M^{2n+1}$ does not admit a  \tri  into $\bC ^N$ for $N=3n$.
	
\end{itemize}

We provide four families of examples according to the residue of the dimension of $M$ modulo $4$.  Let
\[
M^{4k} = \cp2 \times \cdots \times \cp2 = (\cp2)^{\times k}
\]
be the product of $k$ copies of the complex projective plane.

\begin{theorem}\ \label{M}
\begin{itemize}
	\item 
	$M^{4k} $ does not admit a \tri into $\bC ^N$ for $N=6k-1$.
	
	\item 
	$M^{4k+1}=M^{4k}\times S^1 $ does not admit a \tri into 
	$\bC ^N$ for $N=6k$.
	
	\item 
	$M^{4k+2}=M^{4k}\times \rp2 $ does not admit a \tri into 
	$\bC ^N$ for $N=6k+2$.
	
	\item 
	$M^{4k+3}=M^{4k}\times \rp2\times S^1 $ does not admit a \tri into 	$\bC ^N$ for $N=6k+3$.
	
\end{itemize}
\end{theorem}

We first reduce the proof to a statement about bundles.

\begin{lemma}\label{Q}
If a manifold $M$ has a totally real immersion into $\bC^N$, then there exists a complex vector bundle $Q$ over $M$ 
such that 
\[
(\bC \otimes TM)\oplus Q
\]
is trivial of rank $N$.
\end{lemma}

\begin{remark}\label{JL}
This condition in fact characterizes manifolds with totally real immersions into $\bC ^N$ (see \cite{JL}) and includes Theorem \ref{tri}  (see  the appendix).
\end{remark}

\begin{proof}
Let $f:M\to \bC ^N$ be a totally real immersion and define a map 
\[
\psi :\ctm \to T^{1,0}(\bC ^N)
\]
by $\psi (v)=f_*v-iJf_*v$.  It suffices to show that $\psi$ is injective on each fiber.  So let $p\in M$, $\xi$ and $\eta\in TM_p$ and assume that $\psi (\xi +i \eta )=0$.  This implies that
\[
f_*(\xi ) +Jf_*(\eta) =0.
\]
If $\eta \neq 0$, then $f_*(TM)$ would contain the complex line spanned by $f_*(\eta)$ and $Jf_*(\eta)$.  Since $f$ is a totally real immersion, we conclude instead that $\eta,$ and hence also $\xi,$ is zero.
\end{proof}

We now use Chern classes to rewrite the triviality condition of Lemma \ref{Q} in terms of cohomology classes.
Denote the total Chern class of a complex vector bundle $B$ over $M$ by
\[
c(B)=1+c_1(B)+\cdots +c_k(B)
\]
where $c_j(B) \in H^{2j}(M;\bZ)$ and 
$k=\min (\rank B, [\frac{\dim M}{2}])$.  
We use the following properties of Chern classes.  An excellent reference is \cite{MS}.

\begin{enumerate}

	\item 
	(Whitney formula)
	$c(B_1\oplus B_2)=c(B_1)\smallsmile c(B_2)$ where the right hand 	side denotes the cup  product of cohomology classes.
	
	\item If $M$ has a complex structure and dim $M=2n$, then we write
	\[
	c(M)= c(T^{1,0}(M)) = 1+c_1(M)+\cdots +c_n(M).
	\]
	So for such an $M$ 
	\[
	\begin{split}
	\ctm &=c(T^{1,0}(M)\oplus T^{0,1}(M))\\
	&=c(T^{1,0}(M))\smallsmile c(T^{0,1}(M)\\
	&= (1+c_1(M)+\cdots +c_n(M))\\
	&\ \ \smallsmile (1-c_1(M)+\cdots +(-1)^nc_n(M)).
	\end{split}
	\]
	
	\item If $B$ is trivial then $c(B)=1$.
	
\end{enumerate}

Thus if $M$ admits a \tri into $\bC ^N$ then there exists some complex vector bundle $Q$ having rank equal to $N-\dim M$ such that
\begin{equation}\label{trivial}
c(\ctm )\smallsmile c(Q)=1.
\end{equation}
So the proof of Theorem \ref{M} has been reduced to the verification that in the first two cases there is no bundle $Q$ of rank $2k-1$ and in the last two cases no bundle $Q$ of rank $2k$ satisfying \eqref{trivial}.  We present a brief proof of the following well-known result for the sake of completeness.

\begin{lemma}\label{agenerator}
Let $a$ denote the first Chern class of the canonical line bundle on 
$\cp2$.  Then
\[
c(\bC \otimes T\cp2) = 1-3a^2.
\]
\end{lemma}

\begin{proof}
The total Chern class of the complex projective plane is given by 
\[
c(\cp2)=c(T^{1,0}(\cp2))=(1+a)^{3}.
\]
It follows that
\[
c(T^{0,1}(\cp2)) = (1-a)^{3} 
\]
and so
\[
c(\bC \otimes T\cp2)=(1-a^2)^3.
\]
The desired result follows since  $a^3=0$ for dimensional reasons.
\end{proof}

It is known that the Chern class $a$ introduced above generates $H^2(\cp2; \bZ)$.  Similarly the first Chern class of the complexification of the  canonical line bundle $\xi$ of $\rp2$, call it $b$, generates $H^2(\rp2; \bZ)$ (this cohomology group is isomorphic to $\bZ_2,$ a cyclic group of order $2$).  Indeed, the mod $2$ reduction of $b$ is the second Stiefel-Whitney class $w_2(2\xi)$ of twice the canonical line bundle, so is equal to the square of $w_1(\xi)$, which is nonzero.

Let $M$ be one of the manifolds in Theorem \ref{M}.  Let $a_1,\ldots ,a_k$ be the pull-backs of $a$ to $M$ under the corresponding projections to $\cp2$, so that $a_i^3=0$ for all $i$.  Let $b_1$ be the pull-back of $b$ to $M$ for each of the two cases in which $M$  contains a factor $\rp2$. The following result is now clear.

\begin{lemma} \label{c(ctm)}
\begin{eqnarray*}
c(\ctm ^{4k}) \;=\; c(\ctm ^{4k+1}) &=& (1-3a_1^2)\cdots (1-3a_k^2)\\
c(\ctm ^{4k+2}) \;=\; c(\ctm ^{4k+3}) &=& (1-3a_1^2)\cdots (1-3a_k^2)(1+b_1).
\end{eqnarray*}
\end{lemma}

We show first that \eqref{trivial} does not have a solution $Q$ of rank less than $2k$ for $M=M^{4k}$.   
Suppose a complex vector bundle $Q$ satisfies
\[
(1-3a_1^2)\cdots (1-3a_k^2)c(Q)=1.
\]
This implies that 
$c(Q)= (1+3a_1^2)\cdots (1+3a_k^2)$
which, in turn, implies that the rank of $Q$ is at least $2k$ since 
$c_{2k}(Q) = 3^k a_1^2 \cdots a_k^2 \neq 0,$ in view of the K\"unneth formula.  The same argument applies in case $\dim M \equiv 1 \mod 4.$

We next suppose that $M = M^{4k+2}.$  Suppose a complex vector bundle $Q$ satisfies
\[
(1-3a_1^2)\cdots (1-3a_k^2)(1+b_1)c(Q)=1.
\]
This implies that
$c(Q)= (1+3a_1^2)\cdots (1+3a_k^2)(1+b_1)$
which, in turn, implies that the rank of $Q$ is at least $2k+1$ since
$c_{2k+1}(Q) = 3^k a_1^2 \cdots a_k^2 b_1 \neq 0$ in 
$H^{4k+2}(M; \bZ)\cong \bZ_2,$ where we again make use of the 
K\"unneth formula and the fact that the coefficient $3^k$ is odd. The same argument applies in case $\dim M \equiv 3 \mod 4.$

The proof of Theorem \ref{M} is now complete. 


\section{Optimality for independent functions}

Our aim is to show that for each $n>0$, if 
$N > [ \frac{n+1}{2} ]$
then some $n$-manifold $M$ admits no independent mapping of $M$ into 
$\bC^N.$  So Theorem \ref{functions} is also optimal.

Assuming that $F:M \to \bC^N$ is an independent mapping, we extend the differential to a complex linear surjection 
$dF: \bC \otimes T_p^M \to \bC^N$
for each point $p \in M,$  and so obtain a surjective bundle mapping 
$dF: \ctm \to M \times \bC^N.$
Then $K := \ker (dF)$ is a subbundle of $\ctm,$ and therefore 
\[
\ctm \cong K \oplus N\varepsilon.
\]
where $\varepsilon$ denotes a trivial complex line bundle.
It follows that $K$ and $\ctm$ have the same Chern classes.

It should come as no surprise that we will once again use the manifolds appearing in Theorem \ref{M}.

\begin{theorem}\ \label{Mbis}
\begin{itemize}
	\item 
	$M^{4k} $ does not admit an independent mapping to $\bC ^N$ for 	$N>2k$.
	
	\item 
	$M^{4k+1}=M^{4k}\times S^1 $ does not admit an independent mapping 	to $\bC ^N$ for $N > 2k+1$.
	
	\item 
	$M^{4k+2}=M^{4k}\times \rp2 $ does not admit an independent 		mapping to $\bC ^N$ for $N > 2k+1$.
	
	\item 
	$M^{4k+3}=M^{4k}\times \rp2\times S^1 $ does not admit an 			independent mapping to $\bC ^N$ for $N > 2k+2$.	
\end{itemize}
\end{theorem}

\begin{proof}  
For $M^{4k} = (\cp2)^k,$ we have
\[
c_{2k}(\ctm^{4k}) = (-1)^k 3^k a_1^2 \cdots a_k^2 \neq 0
\]
in the notation of Lemma \ref{c(ctm)}.  Hence for an independent mapping $F:M^{4k} \to \bC^N$ we have $c_{2k}(K) \neq 0$ which implies 
$\rank K \geq 2k$ and so $N \leq 2k.$

For $M^{4k+1}$  and an independent mapping $F:M^{4k+1} \to \bC^N$ we again have $\rank K \geq 2k$ and conclude that $N \leq 2k+1.$ 

For $M^{4k+2}$ we have 
\[
c_{2k+1}(\ctm^{4k+2}) = (-1)^k 3^k a_1^2 \cdots a_k^2 b_1 \neq 0
\]
in $H^{4k+2}(M^{4k+2}; \bZ)\cong \bZ_2,$
using the notation of Lemma \ref{c(ctm)}.  Hence for an independent mapping $F:M^{4k+2} \to \bC^N$ we have $\rank K \geq 2k+1$ and conclude that $N \leq 2k+1.$

Finally, for $M^{4k+3}$ and an independent mapping $F:M^{4k+3} \to \bC^N$ we again have $\rank K \geq 2k+1$ and conclude that 
$N \leq 2k+2,$ as desired.   
\end{proof} 
 
Note that if $M$ is a complex manifold and if $F:M\to \bC ^N$ is required to be holomorphic, then the independent maps are precisely the holomorphic submersions of $M$ into $\bC ^N$.  Compare the discussion in the final paragraph of \S 5.


\section{Orientable manifolds of even dimension}

Note that the real projective plane, and the manifolds appearing in Theorem \ref{M} having it as a factor, are not orientable.  On the other hand, every \emph{orientable} $2$-manifold admits a \tri into 
$\bC^2$ (e.g., see \cite[pages 75--76]{lecturenotes} for the case of a compact orientable $2$-manifold; the case of a connected open orientable $2$-manifold is simpler, since then the manifold is parallelizable),
which improves on Theorem \ref{tri}.  We shall generalize this by showing that each orientable closed manifold of dimension $4k+2$ admits a \tri into $\bC^{6k+2}.$  At the same time, our argument allows us to obtain an improved result for orientable $4k$-manifolds having vanishing top dual Pontryagin class (dual Pontryagin classes are defined in the final paragraph of the proof).  

\begin{theorem} \label{tri, orienteven}
Every orientable
$(4k+2)$-manifold $M$ admits a \tri into 
$\bC^{6k+2}.$  Moreover, this result is optimal.  In addition, if an orientable 
$4k$-manifold has vanishing top dual Pontryagin class then it admits a \tri into $\bC^{6k-1}.$
\end{theorem}

\begin{proof}
Let $M$ be an orientable $(4k+2)$-manifold which we assume to be connected, so that $H^{4k+2}(M; \bZ) \cong \bZ$ if $M$ is compact, while this cohomology group vanishes in case $M$ is noncompact (since in the latter case $M$ has the homotopy type of a CW-complex of dimension less than $4k+2$, a well-known result for which a proof is given by Phillips \cite[Lemma 1.1]{Ph}).  By Theorem 1.1 there is a \tri of $M$ into 
$\bC^{6k+3},$ and as a consequence of Lemma 2.2  we have
\[
(\ctm) \oplus Q \,\cong\, (6k+3)\varepsilon
\]
where $\varepsilon$ denotes a trivial complex line bundle and $Q$ is a complex vector bundle of rank $2k+1.$  Let's show that 
$Q \cong Q' \oplus \varepsilon$ for a complex vector bundle $Q'$ of rank $2k.$

We know (\cite[page 158]{MS}) that $c_{2k+1}(Q)$ is equal to the Euler class $e(Q_{\bR})$ of $Q$ viewed as an oriented real vector bundle.  Moreover, this Euler class is the primary obstruction to the existence of a nowhere zero cross-section of $Q_{\bR}$ (\cite[Theorem 12.5]{MS}); in the case we are considering, it is the sole obstruction due to dimensional considerations. So our aim is to show that $c_{2k+1}(Q)$ vanishes.  Now this Chern class can be expressed as a polynomial in the Chern classes 
$c_i(\bC \otimes TM),$ and in each monomial which occurs in this polynomial some index $i$ must be odd, and therefore 
$2c_{2k+1}(Q) = 0$ because $2 c_i(\bC \otimes TM) = 0$ when $i$ is odd (\cite[page 174]{MS}).  Hence $c_{2k+1}(Q)=0$ in 
$H^{4k+2}(M; \bZ)$, since this group is either infinite cyclic or zero.

Replacing $Q$ by $Q' \oplus \varepsilon$ in the formula displayed above, we are in a range in which the trivial line bundle $\varepsilon$ can be cancelled (Husem\"oller presents the details at the start of the chapter ``Stability properties of vector bundles" in his book \textsl{Fibre Bundles} \cite{H}; see the Remark following this proof), 
so we obtain an isomorphism
\[
(\ctm) \oplus Q' \,\cong\, (6k+2)\varepsilon 
\]
which in view of Remark \ref{JL} implies the existence of a \tri of 
$M$ into $\bC^{6k+2}.$

We next point out that $(\cp2)^k \times S^2$ provides an example of an oriented manifold having dimension $4k+2$ which does not admit a \tri into $\bC^{6k+1},$ as follows immediately from the reasoning in the proof of Theorem \ref{M}.  We have therefore found optimal totally real immersions of orientable manifolds having dimensions of the form $4k+2.$

\medskip
Finally, let $M$ be an orientable
$4k$-manifold. We know that there is a \tri of $M$ into $\bC^{6k},$ hence there is a complex vector bundle $Q$ of rank $2k$ for which 
\[
(\ctm) \oplus Q \,\cong\, 6k\varepsilon.
\]
Now let $NM$ denote its normal bundle for an embedding (or immersion) into a Euclidean space, so that $TM \oplus NM$ is trivial. It follows from the Whitney formula that 
$c(Q)$ and $c(\bC \otimes NM)$ are both inverses to $c(\ctm )$ and so are equal to each other.  
By the \emph{dual Pontryagin classes} of $M$ we mean the Pontryagin classes of the normal bundle $NM$, which are equal up to sign with the even Chern classes of $\bC \otimes NM$ and so with the Chern classes $c_{2i}(Q).$  The hypothesis therefore means that the top Chern class $c_{2k}(Q)$ vanishes.  As in the first part of the proof, this implies that $Q \cong Q' \oplus \varepsilon$ for a complex vector bundle $Q'$ of rank $2k-1,$ which in turn implies the existence of a \tri of $M$ into $\bC^{6k-1}.$ 
\end{proof}

\begin{remark} \label{stability}
The result proved by Husem\"oller which was used in the previous argument is the final assertion below.  Let $X$ be an $n$-dimensional CW-complex, and let $\mathrm{Vect}_k(X)$ denote the set of isomorphism classes of $k$-dimensional complex vector bundles over $X$. One defines a map
$\mathrm{Vect}_k(X) \to \mathrm{Vect}_{k+1}(X)$ 
by forming the Whitney sum with the trivial complex line bundle over $X$.  \emph{This map is surjective if $k \geq [ \frac{n}{2}] $, and is bijective if $k \geq [ \frac{n+1}{2}] $.} See also \cite[Theorem 7.3.7]{For2}. 
\end{remark}

We now turn to the analogue of the previous theorem for the case of independent mappings, and obtain similar improvements to Theorem \ref{functions} for orientable closed manifolds of even dimension, as one might be led to anticipate from the comments at the start of this section concerning orientable $2$-manifolds.  The final assertion below is a partial converse to  \cite[Theorem 1.2]{J}.  

\begin{theorem}\label{functions, orienteven}
Every orientable 
$(4k+2)$-manifold $M$ admits an independent mapping to 
$\bC^{2k+2}.$  Moreover, this result is optimal.  In addition, if an orientable 
$4k$-manifold has vanishing top Pontryagin class then it admits an independent mapping to $\bC^{2k+1}.$
\end{theorem}  

\begin{proof}
Let $M$ be an orientable $(4k+2)$-manifold which we assume to be connected, so that $H^{4k+2}(M; \bZ) \cong \bZ$ if $M$ is compact, while this cohomology group vanishes in case $M$ is noncompact.  We know there is an independent mapping of $M$ to $\bC^{2k+1},$ and that consequently we have
\[
\ctm \,\cong\, K \oplus (2k+1)\varepsilon
\]
where 
$K$ is a complex vector bundle of rank $2k+1.$  Let's show that 
$K \cong K' \oplus \varepsilon$ for a complex vector bundle $K'$ of rank $2k.$

As in the proof of Theorem \ref{tri, orienteven}, we need only 
show that $c_{2k+1}(K)$ vanishes.  Now this Chern class 
coincides with $c_{2k+1}(\ctm),$ which has order $2$ (\cite[page 174]{MS}) and so vanishes since it lies in an infinite cyclic group. 

Replacing $K$ by $K' \oplus \varepsilon$ in the formula displayed above, we obtain an isomorphism
\[
\ctm \,\cong\, K' \oplus (2k+2)\varepsilon 
\]
which in view of the analogue of Remark \ref{JL} for independent mappings (see \cite{lecturenotes}) implies the existence of an independent mapping of $M$ into $\bC^{2k+2}.$

We next point out that $(\cp2)^k \times S^2$ provides an example of an oriented manifold having dimension $4k+2$ which does not admit an independent mapping to $\bC^{2k+3},$ as follows immediately from the reasoning in the proof of Theorem \ref{Mbis}.  We have therefore found optimal totally real immersions of orientable manifolds having dimensions of the form $4k+2.$

\medskip
Finally, let $M$ be an orientable 
$4k$-manifold. We know that there is an independent mapping of $M$ to $\bC^{2k},$ hence there is a complex vector bundle $K$ of rank $2k$ for which 
\[
\ctm \,\cong\, K \oplus 2k\varepsilon.
\]
Since $K$ has the same Chern classes as $\ctm$ and the top Pontryagin class of $M$ coincides up to sign with $c_{2k}(\ctm),$ the hypothesis therefore means that the top Chern class $c_{2k}(K)$ vanishes.  As in the first part of the proof, this implies that 
$K \cong K' \oplus \varepsilon$ for a complex vector bundle $K'$ of rank $2k-1,$ which in turn implies the existence of an independent mapping of $M$ to $\bC^{2k+1},$ in view of the analogue of 
Remark \ref{JL} for independent mappings.
\end{proof}

\begin{remark} \label{4k+3}
One knows that every closed orientable $3$-manifold is parallelizable (Stiefel's theorem, e.g. see \cite[Problem 12-B]{MS}).  It is less well known that every open connected orientable $3$-manifold admits an immersion into $\bR^3$ and therefore is parallelizable, which was proved by J.~H.~C. Whitehead \cite{JHCW}.  Hence every orientable 3-manifold admits a totally real immersion into $\bC^3,$ which is also an independent mapping. 

On the other hand, Theorem \ref{M} shows that $\rp2 \times S^1$ does not admit a \tri into $\bC^3$.  From Rudin's result \cite{R} that the Klein bottle admits a totally real embedding into $\bC^2$ it follows that the product of the Klein bottle and $S^1$ is a nonorientable $3$-manifold that does admit a totally real embedding into $\bC^3$. 

The positive results for orientable $3$-manifolds suggest that  improvements to the theorems in \S 1 might be possible for a suitable class of orientable manifolds of dimension $4k+3$.  The best results obtained in this direction assert that for an \emph{open} connected orientable $(4k+3)$-manifold whose stable tangent bundle admits a complex vector bundle structure, a \tri into $\bC^{6k+3}$ and an independent mapping to $\bC^{2k+3}$ exist; in fact, it suffices that all Stiefel-Whitney classes of odd dimension vanish for the tangent bundle.  The key ingredient in the proof of the latter assertion is due to E. Thomas \cite{T} (see also \cite[Problem 15-D]{MS}), who showed that for a real vector bundle $E$ each odd Chern class $c_{2k+1}(\bC \otimes E)$ of its complexification is equal to $\beta(w_{2k}(E) w_{2k+1}(E))$, where $\beta$ denotes the Bockstein coboundary associated to the exact sequence of coefficient groups
\[
0 \to \mathbf{Z} \xrightarrow{2} \mathbf{Z} \to \mathbf{Z}_2 \to 0\,.
\]

\end{remark}


\section{Holomorphic immersions and submersions of Stein manifolds}

We start by recalling the relation of totally real immersions to holomorphic immersions of Stein manifolds.  Doing this allows us to give another proof of Theorem \ref{tri}, as follows.
Whitney showed in \cite{Wh} that any smooth $n$-dimensional manifold $M$ has a compatible real analytic structure.  In the  complexification of this structure, there is  a Stein neighborhood of $M$ \cite{Gra}.  Eliashberg and Gromov proved that any Stein manifold of dimension $n$ admits a holomorphic immersion into 
$\bC ^N$ when $N\geq [\frac {3n} 2]$; see \cite[pages 65--75]{Gro},
\cite [page 151]{For1}, or  \cite [Section 8.5]{For2}.  $M$ is totally real in its Stein neighborhood and so the restriction of a holomorphic immersion of the Stein neighborhood to $M$ is a \tri $M \to \bC^N.$

\medskip
We now show that the manifolds in Theorem \ref{M} yield new examples showing that the target dimension $[\frac {3n} 2]$ for holomorphic immersions of Stein manifolds of complex dimension $n$ is optimal. Observe that a real analytic and totally real immersion extends to a holomorphic immersion of a Stein neighborhood.

\begin{theorem} 
\label{notri}
There exists a Stein manifold of each dimension $n$ that cannot be holomorphically immersed into $\bC ^N$ if $N < [\frac{3n}{2}]$.
\end{theorem}

\begin{proof}
This follows immediately from Theorem \ref{M} and the observation that if $M$ does not have a \tri into $\bC ^N$, then a Stein neighborhood of $M$ in its complexification cannot have a holomorphic immersion into $\bC ^N$.
\end{proof}

Forster \cite{Fo2} (see also \cite{Fo1}) gave the first examples of Stein manifolds satisfying the conclusions of this theorem.  His examples are obtained from the Stein surface
\[
Y = \{[x:y:z] \in \cp2 : x^2+y^2+z^2 \neq 0\},
\] 
by putting  $X^n = Y^{\times m} $ for even $n=2m$, and 
$X^n = Y^{\times m} \times \bC$ for odd $n=2m+1$. 
Forster showed that $Y$ contains $\rp2$ as a deformation retract and a totally real submanifold, and went on to show that the Stein manifolds $X^n$ do not admit holomorphic immersions into $\bC ^N$ for 
$N < [\frac{3n}{2}].$  If one uses the manifolds $(\rp2)^{\times 2k}$ in place of $M^{4k} = (\cp2)^{\times k}$ in Theorem \ref{M}, the proof given there still works (and is essentially the argument due to Forster); but the results presented in \S 4 require examples that are orientable manifolds in dimensions divisible by 4, so we could not use powers of $\rp2$ in these dimensions.  

\medskip
To immerse all smooth manifolds of a given dimension, one expects the target space to be approximately twice the dimension of the manifold.  So a smooth immersion of a manifold of complex dimension $n$ into 
$\bC ^N$ should require that $N$ be roughly $2n.$  The condition of being Stein imposes topological restrictions on the manifold which are reflected in lower immersion dimensions. For example, an easy argument with Stiefel-Whitney classes shows that the Stein manifolds $X^n$ 
used by Forster and discussed briefly above do not even have smooth immersions into the corresponding targets $\bR ^{2N}$, 
when $N < [\frac{3n}{2}]$.  (This can be viewed as an instance of the Oka principle; a problem for suitable holomorphic mappings of a Stein manifold has a solution if and only if the corresponding problem for smooth mappings has a solution.)

\medskip
Finally, we recall the relation of independent mappings to holomorphic submersions of Stein manifolds. Doing this leads to another proof of Theorem 1.2, as follows. As noted at the start of this section, any smooth $n$-dimensional manifold $M$ has a compatible real analytic structure, and the complexification of this structure contains a Stein neighborhood of $M$.  Forstneri\v{c} has proved in \cite [Theorem I]{For1} (see also \cite [Section 8.12]{For2}) that every $n$-dimensional Stein manifold admits 
$[\frac {n+1} {2}]$
 holomorphic functions with pointwise independent differentials, and that this number is maximal for every $n$.  Theorem \ref{functions} follows at once, since holomorphic functions with pointwise independent differentials coincide with independent functions as defined in \S 1.   
In addition, the simple reasoning in the proof of Theorem \ref{notri} immediately yields the maximality asserted in Forstneri\v{c}'s theorem as a consequence of Theorem \ref{Mbis}.  To round out this brief discussion, observe that a holomorphic mapping $f: X \to \bC^N$ with coordinate functions $f_1, \ldots , f_N$ is a holomorphic submersion if and only if its coordinate functions are independent.


\section{Appendix}
We present three simple applications of transversality arguments.

\subsection{Totally real immersions}
We identify $\bC ^N$ with the pair $(\bRtwoN , J)$ where $J:\bRtwoN \to\bRtwoN $ is a linear isomorphism with $J^2=-\textrm{Identity}$.  Then an immersion $f:M\to \bC ^N$ is totally real if for the underlying real map 
\[
f_R :M\to \bRtwoN
\]
we have 
\begin{equation}\label{empty}
f_{R^*}(TM)\cap Jf_{R^*}(TM)=\{ 0\} .
\end{equation}
Let $\onejet$ be the one-jet bundle over $M$.  If $U\subset M$ is a coordinate patch then the restriction of $\onejet$ to $U$ can be coordinatized by
\[
(p,q,a^1,\ldots ,a^n)
\]
where $p\in U$, $q\in \bRtwoN $, $a^j\in \bRtwoN$, and $n=\dim M$.  Note that we think of $q$ as a point in $\bRtwoN$ and each $a^j$ as a column vector.  Denote the $2N	\times n$ matrix $(a^1\cdots a^n)$ by $A$.  If we write, at some point $p\in M$ and using local coordinates
\[
j^1(f)=(p,q,a^1,\ldots ,a^n)
\]
with $a^j=\frac {\partial f}{\partial x_j}$, then the condition that $f$ is an immersion is that $\rank A = n$ and the condition that $f$ is a \tri is that $\rank \, (A,JA)=2n$ where $(A,JA)$ is the $2N\times 2n$ matrix $(a^1\cdots a^n~Ja^1\cdots Ja^n)$.  

We describe a subset $\Sigma \subset \onejet$ by giving it as a subset of $\onejet |_U$ for each $U$ in a coordinate covering of $M$.  Namely
\[
\Sigma =\{(p,q,a^1,\ldots ,a^n):p\in U, q\in \bRtwoN , \rank \,(A,JA)<2n\} .
\]
$\Sigma$ is a stratified subset of $\onejet$ in the sense of \cite{EM}.
We note for later use that $\rank \, (A,JA)$ is even and that when $\rank \, (A,JA)=2n-2$ we may relabel $a^1,\ldots ,a^n$ to obtain that
\[
\{ a^1,\ldots ,a^{n-1},Ja^1,\ldots ,Ja^{n-1}\}
\]
is an independent set.

The first partial derivatives of any smooth map $f:M^n\to \bRtwoN $ determine a section $j^1(f):M\to \onejet $ and the image $j^1(f)(M)$ is a submanifold of dimension $n$.  Clearly, $f$ is a \tri if and only if
\[
j^1(f)(M)\cap \Sigma =\varnothing .
\]
By the simplest case of the Thom Transversality Theorem (see, for example, \cite [page 17]{EM}) any $f:M^n\to \bRtwoN$ (even the constant map) may be perturbed to yield a \tri provided that at a generic point of $\Sigma$ we have
\begin{equation}\label{codim}
\mbox{codim } \Sigma > n .
\end{equation}
Note that at a generic point $\rank \,(A,JA)=n-2$.  So we may assume that the vectors
\[
a^1,\ldots ,a^{n-1},Ja^1,\ldots ,Ja^{n-1}
\]
are independent.  A nearby point $(p^\prime ,q^\prime , b^1,\ldots , b^n)$ is thus in $\Sigma$ exactly when
\begin{equation}\label{b}
b^n\in \mbox{ linear span }\{ b^1,\ldots ,b^{n-1}, Jb^1,\ldots ,Jb^{n-1}\} .
\end{equation}
We complete this latter set to a basis for $\bRtwoN$ and write
\[
b=\sum_1^{n-1} (\alpha _jb^j+\beta _jJb^j)+\sum _1^{2N-2(n-1)}\gamma _ke^k .
\]
We now see that \eqref {b} gives rise to the independent conditions
\[
\gamma _1 =0,\ldots ,\gamma _{2N-2(n-1)} =0 .
\]
So the codimension of $\Sigma $ is $2N-2(n-1)$ and \eqref{codim} holds provided $N\geq [\frac {3n}{2}]$.  This proves Theorem \ref {tri}.

\subsection{Independent maps}

To study independent complex-valued maps $M^n\to \bC ^r$ we write the fibers of $J^1(M,\bC ^r)$ in local coordinates as
\[
J^1(M,\bC ^r)=\{(p,q,\alpha ^1,\ldots , \alpha ^n)\}
\]
where $p\in M$, $q\in \bC ^r$ (thought of as a point), and $\alpha ^j\in \bC ^r $ (thought of as a column vector).  For $F:M\to \bC ^r$ we write
\[
j^1(F)= (p,F(p), \frac {\partial F}{\partial x_1},\ldots , \frac {\partial F}{\partial x_n}).
\]
Previously, we wrote the conditions for $F_1, \ldots , F_n$ to be independent as $dF_1\wedge \cdots \wedge dF_r\neq 0$.  This is the same as requiring that the $r\times n$ matrix
\[
(\frac {\partial F}{\partial x_1}\cdots  \frac {\partial F}{\partial x_n})
\]
has rank $r$.

So now we define $\Sigma$ by
\begin{equation}\label{lindep}
\Sigma =\{(p,q,\alpha ^1,\ldots , \alpha ^n),\,\rank A <r\} 
\end{equation}
where $A$ is the complex $r\times n$ matrix $(\alpha ^1 \cdots  \alpha ^n)$.  More precisely, $\Sigma$ is the subset of $J^1(M,\bC )$ which has \eqref{lindep} as its local expression.  

We seek to compute the codimension of $\Sigma$.  Working at a generic point and relabeling the coordinates of $\bC ^r$ if necessary we assume
\[
\alpha ^1, \alpha ^2,\ldots , \alpha ^{r-1}
\]
are linearly independent and extend to a basis 
\[
\alpha ^1, \alpha ^2,\ldots , \alpha ^{r-1}, e ^1,
e ^2,\ldots ,e^{n-(r-1)}.
\]
For a nearby point $(p^\prime,q^\prime,\beta ^1,\beta ^2,\ldots ,\beta ^r)
$
to be in $\Sigma$ we need that in the complex linear combination
\[
\beta ^r=\sum _1^{r-1} \sigma _j\beta ^j+\sum _1^{n-(r-1)} \gamma _ke ^k
\]
each $\gamma _k$ is zero.  This gives us $2(n-r+1)$ independent real conditions and so this number is the codimension of $\Sigma$.  The condition that codim $\Sigma >n$ becomes
\[
r\leq [\frac {n+1}{2}].
\]
This proves Theorem \ref {functions}.

\subsection{About Lemma \ref{Q}}
As a third example of a transversality calculation we show that if $B$ is a complex vector bundle over a real manifold of dimension $n$ then there is a complex vector bundle $Q$ of rank $[\frac n 2 ]$ such that $B\oplus Q$ is trivial.  We will then use this to relate Lemma \ref{Q} to Theorem \ref{tri}.

\begin{lemma}
Let $B$ be a complex vector bundle of rank $r$ over a manifold $M$ of
dimension $n$. There exists a set of $[n/2]+r$ global sections of $B$ which span  the fiber of $B$ at each point of $M$.
\end{lemma}

\begin{proof}
Let $\rank B=r$, choose a positive integer $k$, and let
\[
{\mathcal B} = B\oplus B\oplus\cdots \oplus B = B^{\oplus k}
\]
be the direct sum of $B$ with itself  $k$ times.   Let
\[
\zeta = (\zeta _1,\ldots ,\zeta _k)
\]
denote a point in the fiber of ${\mathcal B}$ and let $\Sigma $ be the subset of $\mathcal B$ whose fiber over a point  $p\in M$ is given by 
\[
\Sigma |_p = \{ \zeta : \{ \zeta _1,\ldots , \zeta _k\} \mbox{  does not span   } B |_p\}.
\]
At a generic point of $\Sigma$, and after relabeling, $\zeta ^1,\zeta ^2,\ldots ,\zeta ^{r-1}$ and some other  section $e$ may be taken to be a basis for the fibers over a neighborhood of $p$.  For any nearby point $\zeta ^\prime$ we have the linear combinations
\[
\zeta ^{\prime}_j=\sum _{k=1}^{r-1} C_{jk}\zeta _k+\gamma _je\mbox{   for   }j=r,\ldots, k.
\]
So $\Sigma$ is locally defined by the independent complex equations $\gamma _j=0$ and therefore the codimension of $\Sigma$ is $2(k-r+1)$ and  our global spanning sections exist provided $k\geq[\frac n 2 ]+r$.
\end{proof}
Now set $a=[\frac n 2]+r$ and let $\zeta _1,\ldots , \zeta _a$ be global sections of $B$ that span the fiber at each point of $M$.  The map $M\times \bC ^{a} \to B$ given by 
\[
\Lambda (\lambda _1,\ldots ,\lambda _a)=\sum \lambda _j\zeta _j
\]
is surjective.  So we have an isomorphism of bundles
\[
B\oplus Q = M\times \bC ^{n+r}
\]
where $Q$ is the kernel of $\Lambda$.  In particular, there exists $Q$ of rank $[\frac n 2 ]$ so that
\[
(\ctm )\oplus Q
\]
is the trivial bundle of rank $[\frac {3n} 2 ]$.  Thus by the first part of Remark \ref {JL}, $M$ has a \tri into $\bC ^N, N=[\frac {3n} 2]$.  Theorem \ref {tri} then follows.


\bibliographystyle{amsplain}

\end{document}